\documentclass[11pt,a4paper]{article}
\usepackage{latexsym}
\usepackage{amsmath}
\usepackage{amssymb}
\usepackage{amscd}
\usepackage{amsthm}
\usepackage[all]{xy}

\newtheorem{thm}{Theorem}[section]
\newtheorem{cor}[thm]{Corollary}
\newtheorem{lem}[thm]{Lemma}
\newtheorem{prop}[thm]{Proposition}
\newtheorem{pro}[thm]{Property}
\theoremstyle{definition}
\newtheorem{defn}[thm]{Definition}
\newtheorem{rem}[thm]{Remark}

\newcommand{\Z}{\mathbb Z}

\newcommand{\C}{\mathbb C}

\newif\ifpdf \pdftrue
\ifx\pdfoutput\undefined \pdffalse \fi \ifx\pdfoutput\relax
\pdffalse \fi

\usepackage{makeidx}
\makeindex

\begin{document}

\title{A Lefschetz-Riemann-Roch theorem for singular schemes}

\author{Runqiao Fu and Shun Tang}

\date{}

\maketitle

\vspace{-10mm}

\hspace{5cm}\hrulefill\hspace{5.5cm} \vspace{5mm}

\textbf{Abstract.} In this paper, we prove a Lefschetz-Riemann-Roch theorem for singular projective schemes which admit diagonalisable group scheme actions, this result generalizes P. Baum, W. Fulton and G. Quart's Lefschetz-Riemann-Roch theorem for singular varieties (cf. \cite{BFQ}) to general scheme case.



\section{Introduction}
Let $k$ be an algebraically closed field and let $n$ be an integer prime to the characteristic of $k$. A projective $k$-scheme $X$ which admits an automorphism $h_X$ of order $n$ is called an equivariant variety. An equivariant coherent sheaf on $X$ is a coherent sheaf $\mathcal{F}$ on $X$ together with a homomorphism
$$\varphi: h_X^*\mathcal{F}\rightarrow\mathcal{F}.$$ It is clear that this homomorphism induces a family of endomorphisms $H^i(\varphi)$ on cohomology spaces $H^i(X,\mathcal{F})$.

A classical Lefschetz fixed point formula is to give an expression of the alternating sum of the traces of $H^i(\varphi)$, as a sum of the contributions from the components of the fixed point subvariety $X^{h_X}$. On the other hand, roughly speaking, a Lefschetz-Riemann-Roch theorem is a commutative diagram in equivariant $K$-theory which can be regarded as a Grothendieck type generalization of Lefschetz fixed point formula. When $X$ is nonsingular, P. Donovan has proved such a theorem in \cite{Do} by
using the results and some of the methods of the paper of A. Borel and J. P. Serre on the Grothendieck-Riemann-Roch theorem (cf.\cite{BS}). In \cite{BFQ}, P. Baum, W. Fulton and G. Quart generalized Donovan's theorem to singular varieties, the key step of their proof heavily relies on a classic method called the deformation to the normal cone.

Thirteen years later, the results above have been generalized to a new situation by R. W. Thomason in \cite{Th}. In this new situation, the base scheme ${\rm Spec}(k)$ has been replaced by any noetherian, separated and connected scheme $S$ and one generally considers group scheme action on a separated scheme which is of finite type over $S$. To obtain such a Lefschetz fixed point formula, Thomason used the Quillen's localization sequence for higher equivariant $K$-theory, this method helped him to remove the condition of the projectivity for the equivariant schemes and furthermore he can only assume that a morphism between two equivariant schemes is just proper, not necessarily projective. This is very different from a traditional treatment on Grothendieck-Riemann-Roch. If one removes such restrictions, namely if one can accept all necessary conditions about projectivity, he can recover Thomason's formula by using traditional method i.e. the deformation to the normal cone. Actually, our result has some advantages as compared with Thomason's original formula, we shall elaborate these advantages in the last of this paper. The structure of our proof is more or less the same as in \cite{BFQ}.

We now precisely describe our result. Let $N$ be the abelian group $\Z/{n\Z}$ for some $n\in \Z_{+}$, and let $D$ be a noetherian integral regular ring. Define $\mu_n={\rm Spec}(D[N])$ which is a diagonalisable group scheme over $D$. Let $X$ be a scheme over $D$, to give a $\mu_n$-action on $X$ is to give a morphism $m_X: \mu_n\times X\rightarrow X$ which satisfies the usual associativity property. A scheme $X$ which admits a $\mu_n$-action will be called a $\mu_n$-equivariant scheme. Let $Pr_X$ be the canonical projection from $\mu_n\times X$ to $X$. A $\mu_n$-equivariant coherent sheaf on $X$ is a coherent sheaf
$\mathcal{F}$ on $X$ together with an isomorphism $$\varphi:
m_X^*\mathcal{F}\rightarrow Pr_X^*\mathcal{F}$$ which also
satisfies some associativity property (cf. \cite{Ko}). One can image that a $\mu_n$-morphism between two $\mu_n$-equivariant schemes should be understood as a $D$-morphism which is compatible with $\mu_n$-equivariant structures, the definition of $\mu_n$-morphism between two $\mu_n$-equivariant coherent sheaves is similar. Now we say a $\mu_n$-equivariant scheme $X$ is $\mu_n$-projective if there exists a $\mu_n$-equivariant closed immersion $i: X\rightarrow \mathbb{P}_D^r$ which maps $X$ into some projective space endowed with a $\mu_n$-action. One should be
careful about that this definition is more restrictive than the definition given in \cite{Ko}.

Suppose now that we are given a $\mu_n$-action on the free sheaf $\mathcal{F}:=\mathcal{O}_D^{r+1}$ with $r\geq 0$. Using the functorial properties of the ${\rm Proj}$ symbol, we obtain a $\mu_n$-action on the projective space $\mathbb{P}_D^r={\rm Proj}({\rm Sym}(\mathcal{F}))$. A $\mu_n$-action on $\mathbb{P}_D^r$ arising from this way will be called global. The importance of global action is that the twisting sheaf $\mathcal{O}(1)$ of a projective space with global $\mu_n$-action is canonically a $\mu_n$-equivariant invertible sheaf (cf. \cite{Ko}).

Let $X$ be a $\mu_n$-equivariant scheme, by a $\mu_n$-projective envelope of $X$ we understand a $\mu_n$-equivariant closed immersion $j: X\hookrightarrow Z$ which maps $X$ into a $\mu_n$-projective scheme $Z$. When $Z$ is regular, we will call it a $\mu_n$-projective regular envelope. Assume that $X$ and $Y$ are two $\mu_n$-equivariant schemes, a $\mu_n$-morphism $f:
X\rightarrow Y$ is called $\mu_n$-projective if $f$ can be
decomposed as a $\mu_n$-equivariant closed immersion $i:
X\hookrightarrow \mathbb{P}_Y^r$ followed by the canonical
$\mu_n$-projection $p: \mathbb{P}_Y^r\rightarrow Y$.

For a $\mu_n$-equivariant scheme $X$, we denote by $\mathcal{C}_X$ the category of all $\mu_n$-equivariant coherent sheaves on $X$ with $\mu_n$-morphisms, we then denote by $K_0'(X,\mu_n)$ the Grothendieck group of $\mathcal{C}_X$. If $f: X\rightarrow Y$ is a proper $\mu_n$-morphism, we have a push-forward map $$f_*:
K_0'(X,\mu_n)\rightarrow K_0'(Y,\mu_n)$$ which can be defined in a similar way to the non-equivariant case. Define
$R(\mu_n):=K_0(D)[N]$ which is isomorphic to $K_0(D)[T]/{(1-T^n)}$ and which will be fixed as a base ring.

Now assume that $X$ and $Y$ are two integral $\mu_n$-equivariant schemes which both admit a $\mu_n$-projective regular envelope. Moreover, let $f$ be a $\mu_n$-morphism from $X$ to $Y$ which is automatically projective, and let $\mathcal{R}$ be a flat
$R(\mu_n)$-algebra such that $1-T^k$ with $k=1,\ldots,n-1$ are all invertible in $\mathcal{R}$. Our main result reads that there exist homomorphisms $L.$ which are independent of the choice of the envelopes and satisfy the following commutative diagram:
\begin{displaymath}\xymatrix{
	K_0'(X,\mu_n) \ar[r]^-{L.} \ar[d]_{f_*} & K_0'(X_{\mu_n},\mu_n)\otimes_{R(\mu_n)}\mathcal{R} \ar[d]^{{f_{\mu_n}}_*} \\
	K_0'(Y,\mu_n) \ar[r]^-{L.} & K_0'(Y_{\mu_n},\mu_n)\otimes_{R(\mu_n)}\mathcal{R}.}
\end{displaymath}

Here $X_{\mu_n}$ and $Y_{\mu_n}$ stand for the fixed point schemes which are closed subschemes of $X$ and $Y$, and $f_{\mu_n}: X_{\mu_n}\rightarrow Y_{\mu_n}$ is the $\mu_n$-morphism induced by $f$. The authors want to indicate that a fix point formula of Lefschetz type like in \cite[Th\'{e}or\`{e}me 3.5]{Th} can be deduced from this diagram (cf. Corollary~\ref{422}).

\textbf{Acknowledgements.} This work is partially supported by NSFC (no. 12171325) and by National Key R$\&$D Program of China No. 2023YFA1009702.

\section{Preliminaries on equivariant relative $K$-theory}
\subsection{Locally free resolutions of equivariant coherent sheaves}
Let $X$ be a (nice) scheme. The property that any coherent sheaf on $X$ admits a resolution by coherent locally free sheaves is very important for constructing relative $K$-theory. In non-equivariant case, the conditions for this property being true are very simple, for example when $X$ is regular or quasi-projective. The equivariant case looks more complicated because one has to make sure that a $\mu_n$-equivariant coherent sheaf is a quotient of $\mu_n$-equivariant coherent locally free sheaf. The following lemma is the beginning of our argument.

\begin{lem}\label{211}
	Let $X$ be a $\mu_n$-projective scheme, then $X$ admits a
	$\mu_n$-equivariant very ample invertible sheaf relative to its structure morphism.
\end{lem}
\begin{proof}
	This statement is a special case of \cite[Lemma 2.5]{KR}. Although the proof of that lemma heavily depends on the fact that the base scheme is affine, we have the same conclusion for relative projective space over a general noetherian scheme by using the concept of relative ample.
\end{proof}
\begin{prop}\label{212}
	Let $X$ be a $\mu_n$-projective regular scheme, then any
	$\mu_n$-equivariant coherent sheaf on $X$ admits a finite
	resolution by $\mu_n$-equivariant coherent locally free sheaves.
\end{prop}
\begin{proof}
	By Lemma~\ref{211}, there exists a $\mu_n$-equivariant very ample invertible sheaf relative to the structure morphism $f: X\rightarrow {\rm Spec}(D)$, we denote it by $\mathcal{O}(1)$. Then for any $\mu_n$-equivariant coherent sheaf $\mathcal{F}$ on $X$, we have a surjective morphism of coherent sheaves $$f^*f_*(\mathcal{F}\otimes\mathcal{O}(l))\otimes\mathcal{O}(-l)\rightarrow\mathcal{F}\rightarrow
	0\quad\quad(l\gg 0)$$ 
	according to \cite[Thm. 8.8, p. 252]{Ha}. Such a surjective map by construction is equivariant. Let $M$ be a finitely generated and free $\mu_n$-comodule such that there exists a surjective map $M\rightarrow
	f_*(\mathcal{F}\otimes\mathcal{O}(l))$, then we get a surjective map
	$$(f^*\widetilde{M}\otimes\mathcal{O}(-l))\rightarrow\mathcal{F}\rightarrow
	0$$ of equivariant sheaves with
	$(f^*\widetilde{M}\otimes\mathcal{O}(-l))$ coherent locally free. Note that $\mu_n$ is flat so that $m_X^*$ and $Pr_X^*$ are both exact functors, then one may use snake lemma to conclude that the kernel of this surjection is also $\mu_n$-equivariant. Repeating the process above with the kernel of this surjection, we obtain an equivariant coherent locally free resolution of $\mathcal{F}$. Finally, since $X$ is regular, we may conclude that this resolution is finite by a dimension shifting argument.
\end{proof}

Although we don't need the following result, it is theoretically interesting. Its proof will be clear after Property~\ref{2210}.

\begin{cor}\label{213}
	Let notations and assumptions be as above, and denote by
	$K_0(X,\mu_n)$ the Grothendieck group of $\mu_n$-equivariant	coherent locally free sheaves on $X$. Then the nature map $K_0(X,\mu_n)\rightarrow K_0'(X,\mu_n)$ is an isomorphism.
\end{cor}

Given two $\mu_n$-projective schemes $Y_1$ and $Y_2$, we are very interested in whether their fibre product $Y_1\times Y_2$ is also $\mu_n$-projective. There is no reason to believe that this is true for general group scheme action. But we have the following lemma which is enough to meet our requirement.

\begin{lem}\label{214}
	Let $X$ be a $\mu_n$-projective scheme, then the following statements are equivalent:
	
	(a) $X$ admits a $\mu_n$-equivariant closed immersion into some projective space endowed with a global $\mu_n$-action;
	
	(b) There is a $\mu_n$-equivariant very ample invertible sheaf on	$X$.
\end{lem}
\begin{proof}
	This statement is a special case of \cite[Lemma 2.4]{KR}.
\end{proof}
\begin{prop}\label{215}
	Let $\mathbb{P}_D^r$ and $\mathbb{P}_D^s$ be two projective spaces which are both endowed with global $\mu_n$-action, then their fibre product $\mathbb{P}_D^r\times \mathbb{P}_D^s$ is a $\mu_n$-projective scheme.
\end{prop}
\begin{proof}
	By the definition of global $\mu_n$-action, we have
	$\mu_n$-actions on the following two free sheaves
	$\mathcal{F}_1=\mathcal{O}_D^{r+1}$ and
	$\mathcal{F}_2=\mathcal{O}_D^{s+1}$. It follows that we obtain a $\mu_n$-action on their tensor product
	$$\mathcal{F}_1\otimes\mathcal{F}_2=\mathcal{O}_D^{r+1}\otimes\mathcal{O}_D^{s+1}\cong\mathcal{O}_D^{rs+r+s+1}.$$
	This means the projective space $\mathbb{P}_D^{rs+r+s}$ admits a global $\mu_n$-action which is compatible with the Segre embedding
	$\mathbb{P}_D^r\times
	\mathbb{P}_D^s\hookrightarrow\mathbb{P}_D^{rs+r+s}$, so we are done.
\end{proof}

From Lemma~\ref{211}, Lemma~\ref{214} and Proposition~\ref{215} we obtain the following result immediately.

\begin{cor}\label{216}
	Let $Y_1$ and $Y_2$ be two $\mu_n$-projective schemes, then their fibre product $Y_1\times Y_2$ is also a $\mu_n$-projective scheme.
\end{cor}

We now turn to the locally free resolution of a complex of
$\mu_n$-equivariant coherent sheaves on a $\mu_n$-scheme.

\begin{defn}\label{217}
	Let $X$ be a $\mu_n$-scheme, and let $\mathcal{F}.$ be a complex of $\mu_n$-equivariant coherent sheaves on $X$. An equivariant locally free resolution of $\mathcal{F}.$ is a complex of $\mu_n$-equivariant coherent locally free sheaves $\mathcal{E}.$,
	together with a surjective equivariant homomorphism of complexes $\varphi.: \mathcal{E}.\rightarrow\mathcal{F}.$ which is quasi-isomorphism.
\end{defn}

\begin{prop}\label{218}
	Let $X$ be a $\mu_n$-scheme such that any $\mu_n$-equivariant coherent sheaf on $X$ is a quotient of $\mu_n$-equivariant coherent locally free sheaf. Then any complex of $\mu_n$-equivariant coherent sheaves $\mathcal{F}.$ admits a $\mu_n$-equivariant locally free resolution.
\end{prop}
\begin{proof}
	The construction is rather standard in non-equivariant case and also fits the equivariant case. Write $\mathcal{F}.$ as $\cdots\rightarrow F_n\rightarrow\cdots\rightarrow F_0\rightarrow
	0$, by assumption we may choose some surjection $E_0\rightarrow	F_0$ to start. If $\mathcal{E}.$ has been constructed to the $n$-th stage, let $Z_n$ be the kernel of the map $E_n\rightarrow E_{n-1}$ which is $\mu_n$-equivariant, and let $$K_{n+1}={\rm Ker}(Z_n\oplus F_{n+1}\rightarrow F_n)$$ where the map in the brackets takes $(z_n,f_{n+1})$ to
	$\varphi_n(z_n)-d_{n+1}(f_{n+1})$. Choose any surjection of a $\mu_n$-equivariant coherent locally free sheaf $E_{n+1}$ to $K_{n+1}$ to continue the complex one step further. One may check that such a complex $\mathcal{E}.:=\cdots\rightarrow
	E_n\rightarrow\cdots\rightarrow E_0\rightarrow 0$ is what we want.
\end{proof}

\begin{rem}\label{219}
	(i) Any two $\mu_n$-equivariant locally free resolutions can be dominated by a third one;
	
	(ii) If the complex $\mathcal{F}.$ is bounded and all the sheaves	$F_i$ have finite ${\rm Tor}$ dimension i.e. each $F_i$ admits a finite resolution by $\mu_n$-equivariant coherent locally free sheaves. Then the resolving complex $\mathcal{E}.$ can be chosen to be bounded.
\end{rem}

\subsection{Relative $K$-groups and Thom-Gysin homomorphism}
Let $i: X\hookrightarrow Y$ be a $\mu_n$-equivariant closed
immersion, we would like to define the following concept of
equivariant relative $K$-group.

\begin{defn}\label{221}
	For any $\mu_n$-equivariant closed subscheme $X$ of $Y$, we denote by $K_X(Y,\mu_n)$ the free abelian group generated by bounded complexes of homological type $$\mathcal{E}.:\quad 0\rightarrow
	E_n\rightarrow\ldots\rightarrow E_1\rightarrow E_0\rightarrow 0$$
	of $\mu_n$-equivariant coherent locally free sheaves on $Y$, acyclic outside $X$ i.e. for any point outside $X$ the corresponding complex of stalks is exact, modulo the following two kinds of relations:
	
	(i) Given an exact sequence of complexes $$0\rightarrow
	\mathcal{E}'.\rightarrow \mathcal{E}.\rightarrow
	\mathcal{E}''.\rightarrow 0$$ we have
	$[\mathcal{E}.]=[\mathcal{E}'.]+[\mathcal{E}''.]$ in
	$K_X(Y,\mu_n)$.
	
	(ii) Given a quasi-isomorphism from $\mathcal{E}.$ to
	$\mathcal{F}.$ we have $[\mathcal{E}.]=[\mathcal{F}.]$ in
	$K_X(Y,\mu_n)$.
\end{defn}

\begin{rem}\label{222}
	For simplicity, let us denote by $C_X(Y)$ the category of bounded complexes of homological type of $\mu_n$-equivariant coherent locally free sheaves on $Y$ which are acyclic outside $X$. Then any acyclic complex $\mathcal{E}.$ in $C_X(Y)$ represents $0$ in
	$K_X(Y,\mu_n)$ since the natural map $0\rightarrow \mathcal{E}.$ is clearly a quasi-isomorphism.
\end{rem}
\begin{rem}\label{223}
	Let $\mathcal{E}.$ be a complex in $C_X(Y)$ and
	$\mathcal{E}.[-1]=\mathcal{F}.$ the complex such that
	$F_i=E_{i-1}$ and $d_i: F_i\rightarrow F_{i-1}$ equals
	$(-1)^{i-1}$ times $d_{i-1}$. Then we have
	$[\mathcal{E}.[-1]]=-[\mathcal{E}.]$ in $K_X(Y,\mu_n)$. In fact, one may consider the following exact sequence
	$$0\rightarrow \mathcal{E}.\rightarrow \mathcal{C}.\rightarrow \mathcal{E}.[-1]\rightarrow 0,$$
	where $\mathcal{C}.$ is the mapping cone of the identity on $\mathcal{E}.$ i.e.
	$C_i=E_{i-1}\oplus E_i$ with differential $$\left(%
	\begin{array}{cc}
		(-1)^{i-1}d_{i-1} & 0 \\
		{\rm Id}_{E_{i-1}} & d_i \\
	\end{array}%
	\right).$$ One may check that $\mathcal{C}.$ is acyclic which implies that $[\mathcal{C}.]=0$. By induction we can define $\mathcal{E}.[k]$ and we have
	$$[\mathcal{E}.[k]]=(-1)^k[\mathcal{E}.]\quad\text{for all integers}\quad k\leq 0.$$
\end{rem}
\begin{rem}\label{224}
	If $X$ and $X'$ are two $\mu_n$-equivariant closed subschemes of $Y$ such that $X\cap X'$ is a $\mu_n$-equivariant closed subscheme of both $X$ and $X'$. Assume that $\mathcal{E}.$ is in $C_X(Y)$ and $\mathcal{F}.$ is in $C_{X'}(Y)$, then
	$\mathcal{E}.\otimes_{\mathcal{O}_Y}\mathcal{F}.$ is in $C_{X\cap X'}(Y)$. This induces a pairing, called the cup-product, between corresponding relative $K$-groups
	$$\cup:\quad K_X(Y,\mu_n)\otimes K_{X'}(Y,\mu_n)\rightarrow K_{X\cap X'}(Y,\mu_n)$$ with
	$[\mathcal{E}.]\cup [\mathcal{F}.]=[\mathcal{E}.\otimes_{\mathcal{O}_Y}\mathcal{F}.]$.
\end{rem}
\begin{rem}\label{225}
	If $X_i\hookrightarrow Y_i$ $(i=1,2)$ are two $\mu_n$-equivariant closed immersions, we have an external product
	$$\times:\quad K_{X_1}(Y_1,\mu_n)\otimes
	K_{X_2}(Y_2,\mu_n)\rightarrow K_{X_1\times X_2}(Y_1\times
	Y_2,\mu_n)$$ which is given by
	$[\mathcal{E}._1]\times[\mathcal{E}._2]=[\mathcal{E}._1\boxtimes\mathcal{E}._2]$.
\end{rem}
\begin{rem}\label{226}
	Note that if $X$ is a $\mu_n$-equivariant closed subscheme of $Y$, the group $K'_0(X,\mu_n)$ can be identified with the Grothendieck group of the category $M_X(Y)$ of $\mu_n$-equivariant coherent sheaves on $Y$ with supports in $X$. With this observation, let
	$X$ and $X'$ be two $\mu_n$-equivariant closed subschemes of $Y$ which are $\mu_n$-compatible as in Remark~\ref{224}, we may define another important pairing, called the cap-product as follows
	$$\cap:\quad K_X(Y,\mu_n)\times K'_0(X',\mu_n)\rightarrow K'_0(X\cap X',\mu_n)$$
	with $[\mathcal{E}.]\cap[M]=\sum_{i\geq0}(-1)^i[H_i(\mathcal{E}.\otimes_{\mathcal{O}_Y}M)]$.
	In particular, $K'_0(Y,\mu_n)$ has a $K_0(Y,\mu_n)$-module structure.
\end{rem}
\begin{defn}\label{227}
	Let $i: X\hookrightarrow Y$ be a $\mu_n$-equivariant closed immersion, then the cap-product $\cap[\mathcal{O}_Y]$ defined in Remark~\ref{226} induces a map $$h:\quad K_X(Y,\mu_n)\rightarrow
	K_0'(X,\mu_n)$$ which is called the homology map.
\end{defn}
We list some properties of the homology map as follows.
\begin{pro}\label{228}
	With $X_i\hookrightarrow Y_i$ as in Remark~\ref{225}, we have the following commutative diagram
	\begin{displaymath}\xymatrix{
		K_{X_1}(Y_1,\mu_n)\otimes K_{X_2}(Y_2,\mu_n) \ar[rr]^-{\times} \ar[d]_{h\otimes h} && K_{X_1\times X_2}(Y_1\times Y_2,\mu_n) \ar[d]^{h} \\
		K_0'(X_1,\mu_n)\otimes K_0'(X_2,\mu_n) \ar[rr]^-{\times} && K_0'(X_1\times X_2,\mu_n).}
	\end{displaymath}
\end{pro}
\begin{proof}
	This is a standard fact of the homology of the tensor product of two complexes.
\end{proof}

\begin{pro}\label{229}
	Let $X'\hookrightarrow X\hookrightarrow Y$ be a sequence of $\mu_n$-equivariant closed immersions, then we have the following commutative diagram
	\begin{displaymath}\xymatrix{
		K_{X'}(Y,\mu_n) \ar[r]^-{h} \ar[d] & K_0'(X',\mu_n) \ar[d] \\
		K_X(Y,\mu_n) \ar[r]^-{h} & K_0'(X,\mu_n)}
	\end{displaymath}
	where the left vertical map is induced by identity and the right vertical map is induced by the inclusion of $X'$ in $X$.
\end{pro}
\begin{proof}
	The reason is the same as why we can identify $K'_0(X',\mu_n)$ with the Grothendieck group of the category $M_{X'}(Y)$ of $\mu_n$-equivariant coherent sheaves on $Y$ with supports in $X'$.
\end{proof}
\begin{pro}\label{2210}
	If $Y$ is $\mu_n$-projective and regular, then the homology map $h: K_X(Y,\mu_n)\rightarrow K_0'(X,\mu_n)$ is an isomorphism.
\end{pro}
\begin{proof}
	The reason is that any $\mu_n$-equivariant coherent sheaf on $Y$ admits a finite $\mu_n$-equivariant locally free resolution which can be regarded as its inverse image. Remark~\ref{219} implies that this inverse map is well-defined. Then one can use Remark~\ref{223} and an induction argument to prove that this inverse map is surjective.
\end{proof}

To end this section, we give the definition of Thom-Gysin
homomorphism.

\begin{defn}\label{2211}
	Let $X$ be a $\mu_n$-equivariant closed subscheme of $Y$, and let $j: Y\hookrightarrow Z$ be a $\mu_n$-closed immersion with $Z$ regular and $\mu_n$-projective. The Thom-Gysin homomorphism
	$$j_*:\quad K_X(Y,\mu_n)\rightarrow K_X(Z,\mu_n)$$ is defined by $j_*[\mathcal{F}.]=[\mathcal{E}.]$, where $\mathcal{E}.\rightarrow j_*\mathcal{F}.$ is a $\mu_n$-equivariant locally free resolution.
\end{defn}

\begin{prop}\label{2212}
	Let notations and assumptions be as above, then we have the	following commutative diagram
	\begin{displaymath}\xymatrix{
		K_X(Y,\mu_n) \ar[rd]^{h} \ar[dd]_{j_*} & \\ &  \quad K_0'(X,\mu_n)  \\
		K_X(Z,\mu_n) \ar[ru]_{h} &}
	\end{displaymath}
\end{prop}
\begin{proof}
	It is rather clear by the definitions of Thom-Gysin homomorphism and the homology map, since $j$ is a closed immersion and hence $j_*$ is just extension by $0$.
\end{proof}

\subsection{Modified homology map and modified Thom-Gysin homomorphism}
In this section, we will consider equivariant relative $K$-groups of fixed point schemes and define modified homology map and modified Thom-Gysin homomorphism for $\mu_n$-projective regular envelopes. Such a process of modification is very important for constructing the homomorphisms $L.$ promised in the introduction. We first recall some properties of fixed point schemes.
\begin{pro}\label{231}
	Let $Y$ be a $\mu_n$-equivariant scheme, then the fixed point scheme $Y_{\mu_n}$ is a $\mu_n$-equivariant closed subscheme of $Y$. If $Y$ is regular, then $Y_{\mu_n}$ is also regular.
\end{pro}
\begin{proof}
	The first statement comes from \cite[VIII, 6.5 d]{SGA3}. The second one is proved in \cite{Th}.
\end{proof}
\begin{pro}\label{232}
	Let $Y$ be a $\mu_n$-equivariant scheme, then the fixed point scheme $i_{\mu_n}: Y_{\mu_n}\hookrightarrow Y$ is characterized by the following universal property: if $i: Y'\hookrightarrow Y$ is a $\mu_n$-equivariant closed immersion such that the action on $Y'$ is trivial, then there exists a unique closed immersion $j: Y'\hookrightarrow Y_{\mu_n}$ with $i_{\mu_n}\circ j=i$.
\end{pro}
\begin{pro}\label{233}
	Let $Y_1$ and $Y_2$ be two $\mu_n$-schemes with fixed point schemes ${Y_1}_{\mu_n}$ and ${Y_2}_{\mu_n}$ respectively, then the fixed point scheme of $Y_1\times Y_2$ is ${Y_1}_{\mu_n}\times{Y_2}_{\mu_n}$.
\end{pro}
\begin{proof}
	Denote by $Pr_1$ and $Pr_2$ the canonical projections from $Y_1\times Y_2$ to $Y_1$ and $Y_2$ respectively. Assume that $i_0: Y_0\hookrightarrow Y_1\times Y_2$ is a $\mu_n$-closed immersion such that the action on $Y_0$ is trivial, then $Pr_1(Y_0)\subset
	{Y_1}_{\mu_n}$ and $Pr_2(Y_0)\subset{Y_2}_{\mu_n}$. Therefore we get a unique morphism $i: Y_0\rightarrow {Y_1}_{\mu_n}\times {Y_2}_{\mu_n}$. By Property~\ref{232} to conclude the assertion we only have to prove that $i_0=({i_1}_{\mu_n}\times{i_2}_{\mu_n})\circ i$ and this is clear from a diagram chasing argument.
\end{proof}

\begin{pro}\label{234}
	Let $Y$ be a $\mu_n$-equivariant scheme with fixed point scheme $Y_{\mu_n}$, and let $\mathcal{F}$ be a $\mu_n$-equivariant coherent sheaf on $Y$. Then there is a nature $\Z/{n\Z}$-grading $\mathcal{F}\mid_{Y_{\mu_n}}\simeq\sum_{k\in\Z/{n\Z}}\mathcal{F}_k$ on the restriction of $\mathcal{F}$ to
	$Y_{\mu_n}$. Suppose additionally that $Y$ is regular, then the conormal sheaf $N_{Y/{Y_{\mu_n}}}$ is $\mu_n$-equivariant and has vanishing $0$-degree part.
\end{pro}

\begin{pro}\label{235}
	Let $Y$ be a $\mu_n$-equivariant scheme with fixed point scheme $Y_{\mu_n}$, and let $\mathcal{F}$ be a $\mu_n$-equivariant coherent locally free sheaf on $Y_{\mu_n}$. Write
	$$\lambda_{-1}(\mathcal{F}):=\sum_{l=0}^{{\rm
			rk}(\mathcal{F})}(-1)^l\wedge^l(\mathcal{F})\in
	K_0'(Y_{\mu_n},\mu_n),$$ then for any short exact sequence
	$$0\rightarrow \mathcal{F}'\rightarrow
	\mathcal{F}\rightarrow\mathcal{F}''\rightarrow 0$$ we have $\lambda_{-1}(\mathcal{F})=\lambda_{-1}(\mathcal{F}')\cdot\lambda_{-1}(\mathcal{F}'')$.
\end{pro}

\begin{prop}\label{236}
	Let $Y$ be a $\mu_n$-equivariant regular scheme with fixed point scheme $Y_{\mu_n}$, then the element
	$\lambda_{-1}(N_{Y/{Y_{\mu_n}}})$ is invertible in
	$K_0'(Y_{\mu_n},\mu_n)\otimes_{R(\mu_n)} \mathcal{R}$.
\end{prop}
\begin{proof}
	By Property~\ref{234}, $N_{Y/{Y_{\mu_n}}}$ has vanishing
	$0$-degree part, then the statement is a consequence of
	\cite[Lemma 4.5]{KR}.
\end{proof}

\begin{defn}\label{237}
	Let $X$ be a $\mu_n$-scheme, and let $i: X\hookrightarrow Y$ be a $\mu_n$-projective regular envelope which induces a $\mu_n$-projective regular envelope of fixed point scheme $i_{\mu_n}: X_{\mu_n}\hookrightarrow Y_{\mu_n}$. We define the
	modified homology map 
	$$\widetilde{h}:\quad
	K_{X_{\mu_n}}(Y_{\mu_n},\mu_n)\otimes_{R(\mu_n)}\mathcal{R}\rightarrow
	K_0'(X_{\mu_n},\mu_n)\otimes_{R(\mu_n)}\mathcal{R}$$ which maps $\mathcal{E}.\otimes 1$ to
	$i_{\mu_n}^*\lambda_{-1}^{-1}(N_{Y/{Y_{\mu_n}}})\cap
	h(\mathcal{E}.)$. Here $h$ is the homology map defined in
	Definition~\ref{227}.
\end{defn}

Such modified homology maps also satisfy similar properties we list in last section, most of the proofs rely on
Property~\ref{235} and the standard exact sequence relating the normal bundles of a sequence of closed immersions of regular schemes.

\begin{pro}\label{238}
	Let $X_1\subset Y$ and $X_2\subset P$ be two $\mu_n$-projective regular envelopes with $P$ some projective space, then the following diagram commutes:
	\begin{displaymath}\xymatrix{
		K_{{X_1}_{\mu_n}}(Y_{\mu_n},\mu_n)_\mathcal{R}\otimes K_{{X_2}_{\mu_n}}(P_{\mu_n},\mu_n)_\mathcal{R} \ar[rr]^-{\times} \ar[d]_{\widetilde{h}\otimes \widetilde{h}} && K_{{(X_1\times X_2)}_{\mu_n}}((Y_1\times P)_{\mu_n},\mu_n)_\mathcal{R} \ar[d]^{\widetilde{h}} \\
		K_0'({X_1}_{\mu_n},\mu_n)_\mathcal{R}\otimes K_0'({X_2}_{\mu_n},\mu_n)_\mathcal{R} \ar[rr]^-{\times} && K_0'((X_1\times X_2)_{\mu_n},\mu_n)_\mathcal{R}.}
	\end{displaymath}
\end{pro}
\begin{proof}
	From Property~\ref{233} we know that $(Y\times
	P)_{\mu_n}={Y}_{\mu_n}\times P_{\mu_n}$ and $(X_1\times
	X_2)_{\mu_n}={X_1}_{\mu_n}\times {X_2}_{\mu_n}$. Note that $P$ is projective space so that $Y\times P$ and $Y_{\mu_n}\times P$ are both regular, then we may consider the following sequence of regular embeddings $Y_{\mu_n}\times P_{\mu_n}\hookrightarrow
	Y_{\mu_n}\times P\hookrightarrow Y\times P$ which induces a short exact sequence $$0\rightarrow N_{{Y\times P}/{Y_{\mu_n}\times P}}\rightarrow N_{{Y\times P}/{Y_{\mu_n}\times P_{\mu_n}}}\rightarrow N_{{Y_{\mu_n}\times P}/{Y_{\mu_n}\times P_{\mu_n}}}\rightarrow 0.$$ If one remembers that pull-back of ideal sheaves is an operation invariant under base change, one may
	concludes from Property~\ref{235} that
	$$\lambda_{-1}(N_{{Y\times P}/{Y_{\mu_n}\times P_{\mu_n}}})=\lambda_{-1}(N_{Y/{Y_{\mu_n}}})\boxtimes\lambda_{-1}(N_{P/{P_{\mu_n}}})$$
	which covers the only gap between our statement and
	Property~\ref{228}.
\end{proof}

\begin{pro}\label{2310}
	Let $X'\hookrightarrow X\hookrightarrow Y$ be a sequence of $\mu_n$-equivariant closed immersions with $Y$ regular and $\mu_n$-projective, then we have the following commutative diagram
	\begin{displaymath}\xymatrix{
		K_{X'_{\mu_n}}(Y_{\mu_n},\mu_n)_\mathcal{R} \ar[r]^-{\widetilde{h}} \ar[d] & K_0'(X'_{\mu_n},\mu_n)_\mathcal{R} \ar[d] \\
		K_{X_{\mu_n}}(Y_{\mu_n},\mu_n)_\mathcal{R} \ar[r]^-{\widetilde{h}} & K_0'(X_{\mu_n},\mu_n)_\mathcal{R}}
	\end{displaymath}
	where the left vertical map is induced by identity and the right vertical map is induced by the inclusion of $X'_{\mu_n}$ in $X_{\mu_n}$.
\end{pro}
\begin{proof}
	More or less the same as Property~\ref{229}.
\end{proof}

\begin{pro}\label{2311}
	The modified homology map $\widetilde{h}$ is an isomorphism.
\end{pro}
\begin{proof}
	Note that $i_{\mu_n}^*\lambda_{-1}^{-1}(N_{Y/{Y_{\mu_n}}})$ is
	invertible since pull-back is a ring homomorphism, then it follows from Property~\ref{2210}.
\end{proof}

As before, let $Y$ be a $\mu_n$-equivariant scheme with fixed
point scheme $Y_{\mu_n}$. Assume that $\mathcal{F}$ is a
$\mu_n$-equivariant locally free sheaf on $Y$, we will denote by $\mathcal{F}^{(\times)}$ the non-zero degree part of $\mathcal{F}\mid_{Y_{\mu_n}}$.

\begin{defn}\label{2312}
	Let $X$ be a $\mu_n$-scheme, and let $j: Y\rightarrow Z$ be a $\mu_n$-equivariant closed immersion between two
	$\mu_n$-projective regular envelopes of $X$. The modified
	Thom-Gysin homomorphism $$\widetilde{j_*}:\quad
	K_{X_{\mu_n}}(Y_{\mu_n},\mu_n)_\mathcal{R}\rightarrow
	K_{X_{\mu_n}}(Z_{\mu_n},\mu_n)_\mathcal{R}$$ is defined by the formula $\widetilde{j_*}(\mathcal{E}.\otimes
	1)={j_{\mu_n}}_*(\lambda_{-1}N_{Z/Y}^{(\times)}\cdot
	\mathcal{E}.)\otimes 1$ where ${j_{\mu_n}}_*$ is the Thom-Gysin map defined in Definition~\ref{2211}.
\end{defn}

\begin{prop}\label{2313}
	Let notations and assumptions be as above, then we have the following commutative diagram
	\begin{displaymath}\xymatrix{
		K_{X_{\mu_n}}(Y_{\mu_n},\mu_n)_\mathcal{R} \ar[rd]^{\widetilde{h}} \ar[dd]_{\widetilde{j_*}} & \\ &  \quad K_0'(X_{\mu_n},\mu_n)_\mathcal{R}  \\
		K_{X_{\mu_n}}(Z_{\mu_n},\mu_n)_\mathcal{R} \ar[ru]_{\widetilde{h}} &}
	\end{displaymath}
\end{prop}
\begin{proof}
	For simplicity, we omit all necessary notations of
	pull-backs. Consider the following sequence of regular embeddings $Y_{\mu_n}\hookrightarrow Y\hookrightarrow Z$ which induces a short exact sequence $$0\rightarrow N_{Z/Y}\rightarrow N_{Z/{Y_{\mu_n}}}\rightarrow N_{Y/{Y_{\mu_n}}}\rightarrow 0.$$
	Using Property~\ref{235} on the non-zero degree part of this short exact sequence we have
	$$\lambda_{-1}(N_{Z/{Y_{\mu_n}}}^{(\times)})=\lambda_{-1}(N_{Z/Y}^{(\times)})\cdot\lambda_{-1}(N_{Y/{Y_{\mu_n}}})$$
	since the $0$-degree part of $N_{Y/{Y_{\mu_n}}}$ is trivial. Similarly, by considering another sequence of regular embedding
	$Y_{\mu_n}\hookrightarrow Z_{\mu_n}\hookrightarrow Z$, we obtain
	$$\lambda_{-1}(N_{Z/{Y_{\mu_n}}}^{(\times)})=\lambda_{-1}(N_{Z/{Z_{\mu_n}}})\cdot\lambda_{-1}(N_{{Z_{\mu_n}}/{Y_{\mu_n}}}^{(\times)})$$
	since the $0$-degree part of $N_{Z/{Z_{\mu_n}}}$ is trivial. Moreover, by \cite[Lemma 6.7]{KR}, $N_{{Z_{\mu_n}}/{Y_{\mu_n}}}$ has no non-zero degree part and hence $\lambda_{-1}(N_{{Z_{\mu_n}}/{Y_{\mu_n}}}^{(\times)})=1$,so the commutative diagram comes from definitions and
	Proposition~\ref{2212}.
\end{proof}

\section{Deformation to the normal cone}
Assume that $X$ is a $\mu_n$-scheme and that $Y$ is a $\mu_n$-projective regular envelope. Let's define a homomorphism
$$L:\quad K_X(Y,\mu_n)\rightarrow K_{X_{\mu_n}}(Y_{\mu_n},\mu_n)_\mathcal{R}$$ to be the pull-back induced by the inclusion of $Y_{\mu_n}$ in $Y$, followed by the base extension from $R(\mu_n)$ to $\mathcal{R}$. The aim of this section is to prove the following theorem by using the classical
method deformation to the normal cone.

\begin{thm}\label{301}
	Let notations and assumptions be as above. Assume that $Z$ is another $\mu_n$-projective regular envelope of $X$ such that there exists a $\mu_n$-equivariant closed immersion $j: Y\hookrightarrow Z$, then the following diagram commutes.
	\begin{displaymath}\xymatrix{
		K_X(Y,\mu_n) \ar[r]^-{L} \ar[d]_{j_*} & K_{X_{\mu_n}}(Y_{\mu_n},\mu_n)_\mathcal{R} \ar[d]^{\widetilde{j_*}} \\
		K_X(Z,\mu_n) \ar[r]^-{L} & K_{X_{\mu_n}}(Z_{\mu_n},\mu_n)_\mathcal{R}}
	\end{displaymath}
\end{thm}

\subsection{The case of zero section embedding}
Denote by $N_{Z/Y}$ the conormal sheaf of the closed immersion $j: Y\hookrightarrow Z$. Let $P:=\mathbb{P}(N_{Z/Y}\oplus \mathcal{O}_Y)$ be the projective space bundle associated to $N_{Z/Y}\oplus \mathcal{O}_Y$ with the projection $\pi: P\rightarrow Y$. By the universal property of projective space bundle, the surjection $N_{Z/Y}\oplus \mathcal{O}_Y\rightarrow \mathcal{O}_Y$ induces a closed immersion $i: Y\rightarrow P$
which is called the zero section embedding (cf. \cite[Chapter IV]{FL}). In this section, we first prove Theorem~\ref{301} for the zero section embedding above. Consider the following tautological exact sequence 
$$0\rightarrow \mathcal{L}\rightarrow \pi^*(N_{Z/Y}\oplus
\mathcal{O}_Y)\rightarrow\mathcal{O}_P(1)\rightarrow 0$$ 
where $\mathcal{L}$ is the universal hyperplane sheaf on $P$. It is well known that $\mathcal{L}$ admits a canonical regular section which induces a global Koszul resolution $$0\rightarrow \wedge^{{\rm rk}\mathcal{L}}\mathcal{L}\rightarrow\cdots\rightarrow\wedge^3\mathcal{L}\rightarrow
\wedge^2\mathcal{L}\rightarrow\mathcal{L}\rightarrow\mathcal{O}_P\rightarrow
i_*\mathcal{O}_Y\rightarrow 0.$$ 
Note that $i$ is a section of $\pi$ i.e. $\pi\circ i=id_Y$, then for any locally free sheaf $E$ on $Y$ we have $i^*\pi^*E=E$. By projection formula, we obtain
$$i_*E=i_*(i^*\pi^*E\otimes \mathcal{O}_Y)=\pi^*E\otimes
i_*\mathcal{O}_Y.$$ Now let $[\mathcal{E}.]$ be an element in $K_X(Y,\mu_n)$, from Koszul resolution and a little homological arguments one may show that $i_*[\mathcal{E}.]=[\pi^*\mathcal{E}.\otimes\wedge^{\bullet}\mathcal{L}]$ (cf. \cite[Chapter V, Homological Appendix]{FL}). Denote by $i_P$ the inclusion of $P_{\mu_n}$ in $P$, then $L(i_*[\mathcal{E}.])=[i_P^*\pi^*\mathcal{E}.\otimes\wedge^{\bullet}i_P^*\mathcal{L}]\otimes 1$.

On the other hand, denote by $i_Y$ the inclusion of $Y_{\mu_n}$ in $Y$, we have
$$\widetilde{i_*}(L[\mathcal{E}.])=\widetilde{i_*}([i_Y^*\mathcal{E}.]\otimes 1)={i_{\mu_n}}_*(\lambda_{-1}(N_{Z/Y}^{(\times)})\cdot[
i_Y^*\mathcal{E}.])\otimes 1.$$ To calculate the complex class in the last term we need the following lemma to decompose $i_{\mu_n}$.
\begin{lem}\label{311}(cf. \cite[6.3.1]{KR})
	The closed immersion $i_{\mu_n}: Y_{\mu_n}\hookrightarrow P_{\mu_n}$ factors through the $\mu_n$-equivariant closed
	subscheme $\mathbb{P}(N_{Z/Y}^{(0)}\oplus \mathcal{O}_{Y_{\mu_n}})$ where $N_{Z/Y}^{(0)}$ is the $0$-degree part of $N_{Z/Y}\mid_{Y_{\mu_n}}$.
\end{lem}

Denote $\mathbb{P}(N_{Z/Y}^{(0)}\oplus \mathcal{O}_{Y_{\mu_n}})$ by $P_1$. Thanks to this lemma, we may write $i_{\mu_n}$ as $i_1\circ i_0$ where $i_0$ is the closed immersion $Y_{\mu_n}\hookrightarrow P_1$ and $i_1$ stands for the closed immersion $P_1\hookrightarrow P_{\mu_n}$. We first calculate ${i_0}_*(\lambda_{-1}(N_{Z/Y}^{(\times)})\cdot[i_Y^*\mathcal{E}.])=\lambda_{-1}(i_1^*\pi_{\mu_n}^*N_{Z/Y}^{(\times)})\cdot[i_1^*\pi_{\mu_n}^*i_Y^*\mathcal{E}.\otimes {i_0}_*\mathcal{O}_{Y_{\mu_n}}]\in K_{X_{\mu_n}}(P_1,\mu_n)$.

Note that we have the following tautological exact sequence on $P$
$$0\rightarrow \mathcal{L}\rightarrow
\pi^*(N_{Z/Y})\oplus
\mathcal{O}_P\rightarrow\mathcal{O}_P(1)\rightarrow 0,$$ then by restricting to $P_1$ we get a new exact sequence
$$0\rightarrow i_1^*i_p^*\mathcal{L}\rightarrow i_1^*i_p^*\pi^*(N_{Z/Y})\oplus \mathcal{O}_{P_1}\rightarrow\mathcal{O}_{P_1}(1)\rightarrow 0.$$
Taking the non-zero degree part of this exact sequence and using the fact that $i_p^*\pi^*=\pi_{\mu_n}^*i_Y^*$ we obtain an isomorphism $i_1^*\mathcal{L}^{(\times)}\cong i_1^*\pi_{\mu_n}^*(N_{Z/Y}^{(\times)})$. Moreover, the $0$-degree part of the exact sequence above reads that
$\wedge^{\bullet}i_1^*\mathcal{L}^{(0)}$, the pull-back of the $0$-degree part of Koszul complex, gives a resolution of
${i_0}_*\mathcal{O}_{Y_{\mu_n}}$ on $P_1$. So we obtain the following equality
$$\lambda_{-1}(i_1^*\pi_{\mu_n}^*N_{Z/Y}^{(\times)})\cdot[i_1^*\pi_{\mu_n}^*i_Y^*\mathcal{E}.\otimes
{i_0}_*\mathcal{O}_{Y_{\mu_n}}]=[i_1^*i_P^*\pi^*\mathcal{E}.\otimes\wedge^{\bullet}i_1^*i_P^*\mathcal{L}]\in
K_{X_{\mu_n}}(P_1,\mu_n).$$ Since the complex
$\wedge^{\bullet}i_P^*\mathcal{L}$ is acyclic outside $P_1$, we may conclude that
$${i_1}_*[i_1^*i_P^*\pi^*\mathcal{E}.\otimes\wedge^{\bullet}i_1^*i_P^*\mathcal{L}]=[i_P^*\pi^*\mathcal{E}.\otimes\wedge^{\bullet}i_P^*\mathcal{L}]\in K_{X_{\mu_n}}(P_{\mu_n},\mu_n)$$ which completes the proof of Theorem~\ref{301} for the case of zero section embedding.

\subsection{General case}
For general case, let us say $j: Y\hookrightarrow Z$ is a $\mu_n$-equivariant closed immersion between two regular
$\mu_n$-schemes with conormal sheaf $N_{Z/Y}$. We first endow $\mathbb{P}_D^1$ the trivial $\mu_n$-action and consider the
projective line over $Z$:
$$p: \mathbb{P}_Z^1={\rm Proj}(\mathcal{O}_Z[T_0,T_1])\rightarrow Z.$$ Note that $Z={\rm Proj}(\mathcal{O}_Z[T])$, so there are two
canonical sections of $p$ determined by $T_0\rightarrow T, T_1\rightarrow 0$ and $T_0\rightarrow 0, T_1\rightarrow T$ respectively, we denote them by $s_0$ and $s_\infty$.

Let $Y(0)=s_0(j(Y))$ and $Y(\infty)=s_\infty(j(Y))$, we may make the blowing up of $\mathbb{P}_Z^1$ along $Y(\infty)$, that is
$$M={\rm Bl}_{Y(\infty)}(\mathbb{P}_Z^1).$$ It is well known that $P=\mathbb{P}(N_{Z/Y}\oplus \mathcal{O}_Y)$ is the exceptional
divisor for this blowing up, so it determines a closed immersion $g': P\hookrightarrow M$.

Moreover, the embedding of $Y$ in $Z$ induces a closed immersion
of $\mathbb{P}_Y^1$ in $\mathbb{P}_Z^1$. From the composite
\begin{displaymath}\xymatrix{
	Y \ar[r]^-{s_{\infty}} & \mathbb{P}_Y^1 \ar[r] & \mathbb{P}_Z^1}
\end{displaymath}
we get a closed immersion of ${\rm Bl}_{Y(\infty)}(\mathbb{P}_Y^1)$ in ${\rm Bl}_{Y(\infty)}(\mathbb{P}_Z^1)$. Since $Y$ is a Cartier divisor on $\mathbb{P}_Y^1$, ${\rm Bl}_{Y(\infty)}(\mathbb{P}_Y^1)$ is nothing but $\mathbb{P}_Y^1$, so we obtain a closed immersion $$F: \mathbb{P}_Y^1\rightarrow M.$$

Finally, since the blowing down map $\varphi: M\rightarrow \mathbb{P}_Z^1$ is an isomorphism outside $s_\infty(Y)$ and
$s_0(Z)$ is disjoint from $s_\infty(Z)$, the section $s_0$ induces a section $g: Z\rightarrow M.$ Combine the arguments above
together, we have the following deformation diagram which is commutative.
\begin{displaymath}\xymatrix{
	Y \ar[rr]^-{s_0} \ar[d]_{j} && \mathbb{P}_Y^1 \ar[d]^{F} & Y \ar[l]_-{s_\infty} \ar[d]^{i} \\
	Z \ar[rr]^-{g} && M &
	\mathbb{P}(N_{Z/Y}\oplus\mathcal{O}_Y) \ar[l]_-{g'}}
\end{displaymath}

\begin{prop}\label{321}
	The two squares in the deformation diagram are both ${\rm Tor}$-independent.
\end{prop}
\begin{proof}
	Let $f_1$ be the canonical projection from $\mathbb{P}_Y^1$ to $\mathbb{P}_D^1$, and let $f_2$ be the composition of the blowing down map $\varphi$ and the canonical projection from $\mathbb{P}_Z^1$ to $\mathbb{P}_D^1$. It is well known that $f_1$ and $f_2$ are both flat, this fact plays a crucial role in our proof. Assume that the complex of locally free sheaves $\mathcal{L}.$ provides a resolution of
	$F_*\mathcal{O}_{\mathbb{P}_Y^1}$ on $M$. We shall prove that for any point $s\in \mathbb{P}_D^1$, the restriction of $\mathcal{L}.$ to the fibre over $s$ provides a resolution of $\mathcal{O}_Y$. Our statement of this proposition follows from this result because the other component ${\rm Bl}_Y(Z)$ of the fibre over $\infty$ does not meet $Y\times\{\infty\}$. Let $s$ be any point of $\mathbb{P}_D^1$, we get an immersion $k: {\rm Spec}K(s)\rightarrow \mathbb{P}_D^1$. Denote by $j: M_s\rightarrow M$ the fibre of $f_2$ over $s$, then the structure sheaf $\mathcal{O}_{M_s}$ is isomorphic to
	$f_2^{-1}(K(s))\otimes_{j^{-1}f_2^{-1}\mathcal{O}_{\mathbb{P}_D^1}}j^{-1}\mathcal{O}_M$.
	Since $f_1$ and $f_2$ are both flat, $0\rightarrow j^{-1}\mathcal{L}.\rightarrow
	j^{-1}F_*\mathcal{O}_{\mathbb{P}_Y^1}\rightarrow 0$ is an exact sequence of flat
	$j^{-1}f_2^{-1}\mathcal{O}_{\mathbb{P}_D^1}$-modules. By \cite[Lemma A.4.2]{Fu}, the sequence above is also exact after
	tensoring with the $j^{-1}f_2^{-1}\mathcal{O}_{\mathbb{P}_D^1}$-module
	$f_2^{-1}(K(s))$. This means that the pull-back $j^*\mathcal{L}.$ provides a resolution of $\mathcal{O}_Y$ on $M_s$, so we are done.
\end{proof}
According to this proposition, we get the following commutative diagram of relative $K$-groups (cf. \cite[Section 1.6]{BaFM}).
\begin{displaymath}\xymatrix{
	K_X(Y,\mu_n)  \ar[d]_{j_*} & K_{\mathbb{P}_X^1}(\mathbb{P}_Y^1,\mu_n) \ar[l]_-{s_0^*} \ar[r]^-{s_\infty^*} \ar[d]^{F_*} & K_X(Y,\mu_n)  \ar[d]^{i_*} \\
	K_X(Z,\mu_n)  & K_{\mathbb{P}_X^1}(M,\mu_n)\ar[l]_-{g^*} \ar[r]^-{g'^*}&
	K_X(P,\mu_n) }
\end{displaymath}

On the other hand, the deformation diagram constructed above induces the following commutative diagram $(D)$ of fixed point
schemes.
\begin{displaymath}\xymatrix{
	Y_{\mu_n} \ar[rr]^-{{s_0}_{\mu_n}} \ar[d]_{j_{\mu_n}} && \mathbb{P}_{Y_{\mu_n}}^1 \ar[d]^{F_{\mu_n}} & Y_{\mu_n} \ar[l]_-{{s_\infty}_{\mu_n}} \ar[d]^{i_{\mu_n}} \\
	Z_{\mu_n} \ar[rr]^-{g_{\mu_n}} && M_{\mu_n} &
	{\mathbb{P}(N_{Z/Y}\oplus\mathcal{O}_Y)}_{\mu_n} \ar[l]_-{g'_{\mu_n}}}
\end{displaymath}
\begin{prop}\label{322}
	The two squares in the commutative diagram $(D)$ are both ${\rm Tor}$-independent.
\end{prop}
\begin{proof}
	By Lemma~\ref{311}, the closed immersion $i_{\mu_n}$ factors through the closed subscheme $P_1$. The same reason shows that the
	closed immersions from $P_1, \mathbb{P}_{Y_{\mu_n}}^1$ and $Z_{\mu_n}$ to $M_{\mu_n}$ factor through ${\rm
		Bl}_{Y_{\mu_n}(\infty)}\mathbb{P}_{Z_{\mu_n}}^1$. So the ${\rm Tor}$-independence of the commutative diagram $(D)$ is equivalent
	to the ${\rm Tor}$-independence of the deformation diagram on fixed point schemes, then the assertion follows from Proposition~\ref{321}.
\end{proof}

As before, using this proposition, we get a commutative diagram of relative $K$-groups.
\begin{displaymath}\xymatrix{
	K_{X_{\mu_n}}(Y_{\mu_n},\mu_n)_\mathcal{R}  \ar[d]_{\widetilde{j_*}} & K_{\mathbb{P}_{X_{\mu_n}}^1}(\mathbb{P}_{Y_{\mu_n}}^1,\mu_n)_\mathcal{R} \ar[l]_-{{s_0}_{\mu_n}^*} \ar[r]^-{{s_\infty}_{\mu_n}^*} \ar[d]^{\widetilde{F_*}} & K_{X_{\mu_n}}(Y_{\mu_n},\mu_n)_\mathcal{R}  \ar[d]^{\widetilde{i_*}} \\
	K_{X_{\mu_n}}(Z_{\mu_n},\mu_n)_\mathcal{R}  & K_{\mathbb{P}_{X_{\mu_n}}^1}(M_{\mu_n},\mu_n)_\mathcal{R} \ar[l]_-{{g_{\mu_n}}^*} \ar[r]^-{{g'_{\mu_n}}^*}&
	K_{X_{\mu_n}}(P_{\mu_n},\mu_n)_\mathcal{R} }
\end{displaymath}

To complete the whole proof of Theorem~\ref{301}, namely to deform the verification for zero section embedding to general regular
embedding, we still need the following lemma.

\begin{lem}\label{323}
	In the two commutative diagrams of relative $K$-groups we constructed above, $s_0^*$ is surjective and the kernel of
	${g'_{\mu_n}}^*$ is contained in the kernel of ${g_{\mu_n}}^*$.
\end{lem}
\begin{proof}
	The first statement is clear since $s_0$ is a section. For the second statement, one should note that the kernel of
	${s_0}_{\mu_n}^*$ is equal to the kernel of ${s_\infty}_{\mu_n}^*$ by construction and this is also true after tensoring with
	$\mathcal{R}$ since $\mathcal{R}$ is flat over $R(\mu_n)$. Then the correctness of the statement follows easily from the fact that
	all modified Thom-Gysin homomorphisms appearing in the diagram are isomorphisms. This fact can be deduced from Property~\ref{2311}
	and Proposition~\ref{2313}.
\end{proof}

Now we consider the map that $L$ induces between the two commutative diagrams of relative $K$-groups we constructed above,
regarding them as vertical with the fixed point scheme diagram behind the original one. The top and the bottom squares of the
resulting double cubic diagram commute because $L$ commutes with pull-backs. The right side face commutes by our arguments in last
section. So we have
$${g'_{\mu_n}}^*(LF_*-\widetilde{F_*}L)=L{g'}^*F_*-{g'_{\mu_n}}^*\widetilde{F_*}L
=Li_*{s_{\infty}}^*-\widetilde{i_*}{s_{\infty}}_{\mu_n}^*L=(Li_*-\widetilde{i_*}L){s_{\infty}}^*=0.$$
By Lemma~\ref{323}, we get ${g_{\mu_n}}^*(LF_*-\widetilde{F_*}L)=0$ and hence
$(Lj_*-\widetilde{j_*}L)s_0^*=0$. Finally the desired equation $Lj_*-\widetilde{j_*}L=0$ follows from the first statement of
Lemma~\ref{323}.

\section{Main theorem and its proof}
\subsection{The independence of the construction of the morphism $L.$}
Let $X$ be a $\mu_n$-equivariant scheme which may not be regular. Moreover, assume that $X$ admits a $\mu_n$-projective regular
envelope $Y$, then by the arguments we made in last section we may define a homomorphism ${L.}^Y$ from $K_0'(X,\mu_n)$ to
$K_0'(X_{\mu_n},\mu_n)_\mathcal{R}$ as the composition of the following three morphisms
\begin{displaymath}\xymatrix{
	K_0'(X,\mu_n) \ar[r]^-{h^{-1}} & K_X(Y,\mu_n) \ar[r]^-{L} &
	K_{X_{\mu_n}}(Y_{\mu_n},\mu_n)_\mathcal{R} \ar[r]^-{\widetilde{h}}
	& K_0'(X_{\mu_n},\mu_n)_\mathcal{R}.}
\end{displaymath}

In this section, we shall prove that actually the homomorphism defined above is independent of the choice of the envelope so that
we can remove $Y$ from the notation and just write ${L.}$. This fact is also the core of our proof of the Lefschetz-Riemann-Roch
theorem in next subsection.
\begin{prop}\label{411}
	Let notations and assumptions be as above. Suppose that $Z$ is another $\mu_n$-projective regular envelope of $X$ such that there exists a $\mu_n$-equivariant closed immersion $j: Y\hookrightarrow Z$, then ${L.}^Y={L.}^Z$.
\end{prop}
\begin{proof}
	The statement is equivalent to the commutativity of the following diagram
	\begin{displaymath}\xymatrix{
		& K_X(Y,\mu_n) \ar[r]^-{L} \ar[dd]_{j_*} & K_{X_{\mu_n}}(Y_{\mu_n},\mu_n)_\mathcal{R} \ar[rd]^{\widetilde{h}} \ar[dd]^{\widetilde{j_*}} & \\
		K_0'(X,\mu_n) \ar[ru]^{h^{-1}} \ar[rd]^{h^{-1}} & & & K_0'(X_{\mu_n},\mu_n)_\mathcal{R}. \\
		& K_X(Z,\mu_n) \ar[r]^-{L} & K_{X_{\mu_n}}(Z_{\mu_n},\mu_n)_\mathcal{R} \ar[ru]^{\widetilde{h}} & }
	\end{displaymath}
	
	In fact, these two triangles commute by Property~\ref{2210}, Proposition~\ref{2212} and Proposition~\ref{2313}. The square
	commutes by Theorem~\ref{301}.
\end{proof}

\begin{prop}\label{412}
	Let $X_1\subset Y$ and $X_2\subset P$ be two $\mu_n$-projective regular envelopes with $P$ some projective space, then the
	following diagram commutes:
	\begin{displaymath}\xymatrix{
		K_0'(X_1,\mu_n)\otimes K_0'(X_2,\mu_n) \ar[rr]^-{{L.}^Y\otimes{L.}^P} \ar[d]_{\times} && K_0'({X_1}_{\mu_n},\mu_n)_\mathcal{R}\otimes K_0'({X_2}_{\mu_n},\mu_n)_\mathcal{R} \ar[d]^{\times} \\
		K_0'(X_1\times X_2,\mu_n) \ar[rr]^-{{L.}^{Y\times P}} && K_0'((X_1\times X_2)_{\mu_n},\mu_n)_\mathcal{R}.}
	\end{displaymath}
\end{prop}
\begin{proof}
	This comes from Property~\ref{228}, Property~\ref{238} and the fact that $L$ commutes with products.
\end{proof}

\begin{prop}\label{413}
	Let $X$ be a $\mu_n$-scheme which admits a $\mu_n$-projective
	regular envelope $Y$. Suppose that $f: X'\hookrightarrow X$ is a
	$\mu_n$-equivariant closed immersion, then the following diagram
	commutes:
	\begin{displaymath}\xymatrix{
		K_0'(X',\mu_n) \ar[r]^-{{L.}^Y} \ar[d]_{f_*} & K_0'(X'_{\mu_n},\mu_n)_\mathcal{R} \ar[d]^{{f_{\mu_n}}_*} \\
		K_0'(X,\mu_n) \ar[r]^-{{L.}^Y} & K_0'(X_{\mu_n},\mu_n)_\mathcal{R}.}
	\end{displaymath}
\end{prop}
\begin{proof}
	This diagram is the outside of the following diagram
	\begin{displaymath}\xymatrix{
		K_0'(X',\mu_n) \ar[r]^-{h^{-1}} \ar[d]_{f_*} & K_{X'}(Y,\mu_n)
		\ar[r]^-{L} \ar[d] & K_{X'_{\mu_n}}(Y_{\mu_n},\mu_n)_\mathcal{R}
		\ar[r]^-{\widetilde{h}} \ar[d] &
		K_0'(X'_{\mu_n},\mu_n)_\mathcal{R} \ar[d]^{{f_{\mu_n}}_*} \\
		K_0'(X,\mu_n) \ar[r]^-{h^{-1}} & K_X(Y,\mu_n) \ar[r]^-{L} &
		K_{X_{\mu_n}}(Y_{\mu_n},\mu_n)_\mathcal{R} \ar[r]^-{\widetilde{h}}
		& K_0'(X_{\mu_n},\mu_n)_\mathcal{R}.}
	\end{displaymath}
	The middle square is naturally commutative. The commutativity of the first and the third squares comes from Property~\ref{229}, Property~\ref{2310} and the fact that the equivariant $K$-groups are generated by $f$-acyclic elements.
\end{proof}
Before we can give two more propositions, we need a technical lemma as follows.
\begin{lem}\label{414}
	Let $X$ be a $\mu_n$-scheme, and let $P$ be some $\mu_n$-projective space endowed with a global $\mu_n$-action.
	Then the external product map
	\begin{displaymath}\xymatrix{
		K_0'(X,\mu_n)\otimes K_0'(P,\mu_n) \ar[r]^-{\times} & K_0'(X\times
		P,\mu_n)}
	\end{displaymath}
	is an isomorphism.
\end{lem}
\begin{proof}
	A simple proof of this statement can be obtained by following \cite[Section 8]{Qu1}. Denote by $S$ the external product map, we give an explicit proof of the surjectivity of $S$ since we only need this fact in this paper. Actually, Quillen has shown that $K_0'(X\times P,\mu_n)$ is generated by equivariant regular coherent sheaves on $X\times P$ and every regular sheaf
	$\mathcal{F}$ admits a canonical resolution
	$$0\rightarrow T_N(\mathcal{F})\boxtimes\mathcal{O}_P(-N)\rightarrow\cdots\rightarrow T_0(\mathcal{F})\boxtimes\mathcal{O}_P\rightarrow \mathcal{F}\rightarrow 0$$ where $T_k(\mathcal{F})$ are coherent
	sheaves on $X$. The construction of these exact functors $T_k$ are equivariant so that if $\mathcal{F}$ is equivariant then all
	$T_k(\mathcal{F})$ are equivariant. Note that we have assumed that $P$ admits a global $\mu_n$-action so that all $\mathcal{O}_P(-k)$
	are equivariant, therefore for any equivariant regular coherent sheaf $\mathcal{F}$ on $X\times P$ the canonical resolution above also fits the equivariant setting. Hence the canonical resolution implies that for any equivariant regular coherent sheaf $\mathcal{F}$ on $X\times P$ we have
	$$\mathcal{F}=S(\sum_{k=0}^N(-1)^kT_k(\mathcal{F})\otimes
	O_P(-k)).$$ This means $S$ is surjective. For the injectivity, one can follow the same argument (with a little change) as in \cite[Theorem 2.1]{Qu1}. This argument was also formulated in \cite[Theorem 3.1]{Th1}.
\end{proof}
\begin{prop}\label{415}
	Let $X$ be a $\mu_n$-scheme which admits a $\mu_n$-projective regular envelope $Y$, and let $P$ be some projective space endowed with a global $\mu_n$-action. Denote by $p: X\times P\rightarrow X$ the canonical projection and assume that the diagram
	\begin{displaymath}\xymatrix{
		K_0'(P,\mu_n) \ar[r]^-{{L.}^P} \ar[d]_{q_*} & K_0'(P_{\mu_n},\mu_n)_\mathcal{R} \ar[d]^{{q_{\mu_n}}_*} \\
		K_0'(D,\mu_n) \ar[r]^-{{L.}^D} & K_0'(D,\mu_n)_\mathcal{R}}
	\end{displaymath}
	commutes, then we have the following commutative diagram
	\begin{displaymath}\xymatrix{
		K_0'(X\times P,\mu_n) \ar[r]^-{{L.}^{Y\times P}} \ar[d]_{p_*} & K_0'(X_{\mu_n}\times P_{\mu_n},\mu_n)_\mathcal{R} \ar[d]^{{p_{\mu_n}}_*} \\
		K_0'(X,\mu_n) \ar[r]^-{{L.}^Y} & K_0'(X_{\mu_n},\mu_n)_\mathcal{R}.}
	\end{displaymath}
	This property is called the base change invariance.
\end{prop}
\begin{proof}
	By our assumption, we may consider the following commutative
	diagram
	\begin{displaymath}\xymatrix{
		K_0'(X,\mu_n)\otimes K_0'(P,\mu_n) \ar[rr]^-{{L.}^Y\otimes{L.}^P} \ar[d]_{1\otimes q_*} && K_0'(X_{\mu_n},\mu_n)_\mathcal{R}\otimes K_0'(P_{\mu_n},\mu_n)_\mathcal{R} \ar[d]^{1\otimes {q_{\mu_n}}_*} \\
		K_0'(X,\mu_n)\otimes K_0'(D,\mu_n) \ar[rr]^-{{L.}^Y\otimes{L.}^D} && K_0'(X_{\mu_n},\mu_n)_\mathcal{R}\otimes K_0'(D,\mu_n)_\mathcal{R}.}
	\end{displaymath}
	The external product maps this square to the desired one. These two sides of the resulting cubic diagram commute by the fact that external product is compatible with push-forwards of proper maps. The top and the bottom squares commute by Proposition~\ref{412}. Then the desired commutativity of the back square follows from Lemma~\ref{414}
\end{proof}
The following result shows that the assumption in Proposition~\ref{415} is always true.
\begin{prop}\label{416}
	Let $q: P\rightarrow D$ be the structure morphism of some projective space which is endowed with a global $\mu_n$-action.
	Then the following diagram commutes:
	\begin{displaymath}\xymatrix{
		K_0'(P,\mu_n) \ar[r]^-{{L.}^P} \ar[d]_{q_*} & K_0'(P_{\mu_n},\mu_n)_\mathcal{R} \ar[d]^{{q_{\mu_n}}_*} \\
		K_0'(D,\mu_n) \ar[r]^-{{L.}^D} & K_0'(D,\mu_n)_\mathcal{R}.}
	\end{displaymath}
\end{prop}
\begin{proof}
	By the base change invariance i.e. Proposition~\ref{415}, we only have to show that the statement is correct when $D$ is equal to $\Z$. In fact, since $\Z$ is a ${\rm PID}$, every locally free sheaf on $\Z$ is free, then $K_0(\Z)$ can be identified with $K_0(\C)$. So Lemma~\ref{414} tells us that the base change maps $K_0'(\mathbb{P}_\Z^r,\mu_n)\rightarrow
	K_0'(\mathbb{P}_\C^r,\mu_n)$ and $K_0'(\Z,\mu_n)\rightarrow K_0'(\C,\mu_n)$ are isomorphisms. Moreover, note that the fixed
	point scheme of $\mathbb{P}_\Z^r$ is the disjoint union of some projective spaces endowed with global $\mu_n$-actions (cf.
	\cite[6.3.1]{KR}), then by using Lemma~\ref{414} again we know that the base change map $K_0'({\mathbb{P}_\Z^r}_{\mu_n},\mu_n)\rightarrow K_0'({\mathbb{P}_\C^r}_{\mu_n},\mu_n)$ is also isomorphism.
	Therefore we are reduce to prove that the diagram is commutative when $\Z$ is replaced by $\C$, but this is a slight variant of the main result of \cite{BFQ}, so we are done.
\end{proof}
\begin{thm}\label{417}
	Let notations and assumptions be as at the beginning of this section, then the homomorphism $L.^Y$ we constructed there is
	independent of the choice of the envelope.
\end{thm}
\begin{proof}
	Suppose that $Y_1$ and $Y_2$ are two $\mu_n$-projective regular envelopes of $X$, we have to show that $L.^{Y_1}=L.^{Y_2}$. According to Proposition~\ref{411} and Lemma~\ref{214}, we may assume that $Y_1$ and $Y_2$ are two projective spaces which are both endowed with global $\mu_n$-actions. Consider the diagonal embedding $i: X\rightarrow X\times Y_2$ given by $x\mapsto (x,x)$, the canonical projection $p: X\times Y_2\rightarrow X$, and the following diagram
	\begin{displaymath}\xymatrix{
		K_0'(X,\mu_n) \ar[rr]^-{{L.}^{Y_1\times Y_2}} \ar[d]_{i_*} &&
		K_0'(X_{\mu_n},\mu_n)_\mathcal{R} \ar[d]^{{i_{\mu_n}}_*}\\
		K_0'(X\times Y_2,\mu_n) \ar[rr]^-{{L.}^{Y_1\times Y_2}} \ar[d]_{p_*} && K_0'(X_{\mu_n}\times {Y_2}_{\mu_n},\mu_n)_\mathcal{R} \ar[d]^{{p_{\mu_n}}_*} \\
		K_0'(X,\mu_n) \ar[rr]^-{{L.}^{Y_1}} && K_0'(X_{\mu_n},\mu_n)_\mathcal{R}.}
	\end{displaymath}
	The upper square commutes by Proposition~\ref{413} and the lower square commutes by Proposition~\ref{415}, so the outside square is commutative. But by construction $p_*i_*$ is the identity map on $K_0'(X,\mu_n)$, therefore $L.^{Y_1\times Y_2}$ is equal to $L.^{Y_1}$. By a symmetric argument we finally have $L.^{Y_1}=L.^{Y_2}$, so we are done.
\end{proof}

\subsection{Lefschetz-Riemann-Roch theorem}
\begin{thm}\label{421}
	Let $X$ and $Y$ be two $\mu_n$-schemes which both admit a
	$\mu_n$-projective regular envelope, and let $f: X\rightarrow Y$
	be a $\mu_n$-equivariant morphism. Then the following diagram
	\begin{displaymath}\xymatrix{
		K_0'(X,\mu_n) \ar[r]^-{L.} \ar[d]_{f_*} & K_0'(X_{\mu_n},\mu_n)\otimes_{R(\mu_n)}\mathcal{R} \ar[d]^{{f_{\mu_n}}_*} \\
		K_0'(Y,\mu_n) \ar[r]^-{L.} & K_0'(Y_{\mu_n},\mu_n)\otimes_{R(\mu_n)}\mathcal{R}.}
	\end{displaymath}
	commutes. When $X$ is regular,
	$L.[\mathcal{O}_X]=\lambda_{-1}^{-1}(N_{X/{X_{\mu_n}}})\cap [\mathcal{O}_{X_{\mu_n}}]$.
\end{thm}
\begin{proof}
	Since $X$ is $\mu_n$-projective, it admits a regular envelope $P$ which is some projective space endowed with a global $\mu_n$-action. Then the morphism $f$ factors into a closed
	immersion $X\hookrightarrow Y\times P$ followed by a projection
	$Y\times P\rightarrow Y$. So the commutative diagram in this theorem comes from Proposition~\ref{413}, Proposition~\ref{415}
	and the independence of the construction of the morphism $L.$ i.e. Theorem~\ref{417}. The second statement is clear when one chooses $X$ itself as a $\mu_n$-projective regular envelope.
\end{proof}

\begin{cor}\label{422}
	Let notations and assumptions be as above, and let $j:X\hookrightarrow Z$ be any $\mu_n$-projective regular envelope of $X$. Then for any $\mu_n$-equivariant coherent sheaf $\mathcal{F}$ on $X$ the following fixed point formula of
	Lefschetz type holds in $K_0'(Y_{\mu_n},\mu_n)_\mathcal{R}$:
	$$L.(\sum_{k\geq0}(-1)^kR^kf_*(\mathcal{F}))={f_{\mu_n}}_*((\lambda_{-1}^{-1}(N_{Z/{Z_{\mu_n}}}))\cap\sum_{l\geq0}(-1)^l{\rm Tor}^l_{\mathcal{O}_Z}(j_*\mathcal{F},\mathcal{O}_{Z_{\mu_n}})).$$
\end{cor}

\begin{rem}\label{423}
	Comparing with \cite[Th\'{e}or\`{e}me 3.5]{Th}, the formula in Corollary~\ref{422} has at least two advantages if $Y$ admits a
	$\mu_n$-projective regular envelope:
	
	(i) We don't require that the morphism $f: X\rightarrow Y$ factors through the envelope $Z$;
	
	(ii) When $Y$ is not regular, our morphism $L.$ from $K'_0(Y,\mu_n)\rightarrow K'_0(Y_{\mu_n},\mu_n)_{\mathcal{R}}$ is constructible by choosing a projective regular envelope for $Y$.
\end{rem}

\hspace{5cm} \hrulefill\hspace{5.5cm}

Runqiao Fu

iFLYTEK Co., Ltd., No.666 West Wangjiang Road Hefei, Anhui, China\\

Shun Tang

Academy for Multidisciplinary Studies \& School of Mathematical Sciences, Capital Normal University, West 3rd Ring North Road 105,
100048 Beijing, P. R. China

E-mail: tangshun@cnu.edu.cn

\end{document}